\documentclass[12pt]{amsart}
\usepackage[centertags]{amsmath}
\usepackage{amscd}
\usepackage{amsfonts}
\usepackage{amssymb}
\usepackage{amsthm}
\usepackage{newlfont}
\usepackage{mathrsfs}
\usepackage[T1]{fontenc}
\usepackage{anysize}

\usepackage[margin=2cm]{geometry}
\theoremstyle{plain}
\newtheorem{thm}{Theorem}[section]
\newtheorem{lem}[thm]{Lemma}
\newtheorem{prop}[thm]{Proposition}
\newtheorem{cor}[thm]{Corollary}

\newcommand{\thmref}[1]{Theorem~\ref{#1}}
\newcommand{\lemref}[1]{Lemma~\ref{#1}}

\newtheorem{rmk}[thm]{Remark}
\theoremstyle{definition}

\theoremstyle{remark}
\newtheorem{rem}{Remark}[section]
\theoremstyle{remark}

\numberwithin{equation}{section}

\newcommand{\x}{\textbf}
\newcommand{\q}{\quad}
\newcommand{\qq}{\qquad}
\newcommand{\mbb}{\mathbb}
\newcommand{\p}{\prime}
\newcommand{\mc}{\mathcal}
\newcommand{\mrm}{\mathrm}

\begin{document}

\title{Note on Hermitian Jacobi forms}

\author{SOUMYA DAS}
\address{Harish Chandra Research Institute\\ 
         Chhatnag Road\\  
         Jhusi Allahabad 211019, India.}
\email{somu@hri.res.in, soumya.u2k@gmail.com}

\begin{abstract}
We compare the spaces of Hermitian Jacobi forms (HJF) of weight $k$ and indices $1,2$ with classical Jacobi forms (JF) of weight $k$ and indices $1,2,4$. Using the embedding into JF, upper bounds for the order of vanishing of HJF at the origin is obtained. We compute the rank of HJF as a module over elliptic modular forms and prove the algebraic independence of the generators in case of index $1$. Some related questions are discussed.
\end{abstract}

\maketitle

\section{Introduction}

Hermitian Jacobi forms of integer weight and index are defined for the Jacobi group over the ring of integers $\mc{O}_{K}$ of an imaginary quadratic field $K$. They were defined and studied by K. Haverkamp in \cite{haverkamp/en}. In \cite{somu} differential operators were constructed from the Taylor expansion of Hermitian Jacobi forms in analogy to that for classical Jacobi forms in \cite{zagier} and a certain subspace of Hermitian Jacobi forms was realized as a subspace of a direct product of elliptic modular forms for the full modular group. The structural properties of index $1$ froms were treated in \cite{sasaki}.

In this paper we treat classical Jacobi forms as an intermediate space between Hermitian Jacobi forms and elliptic modular forms. We present some of the structural properties of index $2$ forms using the restriction maps $\pi_{\rho} \colon J_{k,m}(\mathcal{O}_{K} ) \rightarrow J_{k, N(\rho)m}$ defined by  $ \pi_{\rho} \phi(\tau,z_{1},z_{2}) = \phi(\tau,\rho z,\bar{\rho} z)$ ($\rho \in \mc{O}_{K}$, see \cite{haverkamp/en}). Since we do not have (at present) the order of vanishing of a Hermitian Jacobi forms at the origin (which is known to be $2m$, $m$ being the index; in the case of classical Jacobi forms), computations involving the Taylor expansions is not very fruitful for $m \geq 2$.

The main results are in section~\ref{index1},~\ref{index2} and~\ref{rank}. The purpose of this note is to look at the structure of index $2$ forms by comparing them with classical Jacobi forms. We also compute the rank of index $m$ forms of weight a multiple of $2$ and $4$ (denoted as $J_{n*,m}(\mathcal{O}_{K}), \, n=2,4$) as a module over the algebra of elliptic modular forms. Unlike the classical Jacobi forms, the number of homogeneous products of degree $m$ of the index $1$ generators is less than the rank. In the final section, we discuss several related questions on Hermitian Jacobi forms.

\section{Notations and definitions}  

\subsection{Hermitian Jacobi forms}

Let $\mathcal{H}$ be the upper half plane. Let $K = \mbb{Q}(i)$ and $\mathcal{O}_{K} = \mbb{Z}[i]$ be it's ring of integers.
Let $\Gamma_{1}(\mathcal{O}_{K} ) = \left \{ \epsilon M \mid M \in SL(2,\mbb{Z}) , \epsilon \in \mathcal{O}^{\times}_{K} \right \}$. The Jacobi group over $\mathcal{O}_{K}$ is $\Gamma^{J}(\mathcal{O}_{K} ) = \Gamma_{1}(\mathcal{O}_{K} ) \ltimes \mathcal{O}_{K}^{2}.$ \\
The space of Hermitian Jacobi forms for $\Gamma^{J}(\mathcal{O}_{K} )$ of weight $k$ and index $m$, where $k$, $m$ are positive  integers, consists of holomorphic functions $\Phi$ on $ \mathcal{H} \times \mbb{C}^{2}$ satisfying : 
\begin{eqnarray}
\numberwithin{equation}{section}
\label{jacobi1} \q \phi(\tau,z_{1},z_{2}) = \phi|_{k,m} \epsilon M(\tau,z_{1},z_{2}) := \epsilon^{-k}(c\tau + d)^{-k} e^{\frac{-2\pi i m c z_{1}z_{2}}{c\tau + d}} \phi \left(M\tau, \frac{\epsilon z_{1}}{c\tau + d},\frac{\bar{\epsilon} z_{2}}{c\tau + d}\right) \\
\mbox{ for all }  M = \left( \begin{smallmatrix}
a & b \\
c & d  \end{smallmatrix}\right) \mbox{ in } SL(2,\mbb{Z}) \nonumber \\
\label{jacobi2} \phi(\tau,z_{1},z_{2}) = \phi|_{k,m}[\lambda,\mu] := e^{2 \pi i m (N(\lambda)\tau + \bar{\lambda}z_{1}+\lambda z_{2})}
\phi (\tau,z_{1} + \lambda \tau + \mu, z_{2} + \bar{\lambda} \tau + \bar{\mu}) \\
\mbox{ for all }  \lambda , \mu \mbox{ in }  \mathcal{O}_{K}. \nonumber 
\end{eqnarray}

The (finite dimensional) complex vector space of Hermitian Jacobi forms of weight $k$ and index $m$ is denoted by $J_{k,m}(\mathcal{O}_{K} ).$ Such a form has a Fourier expansion : 
\begin{eqnarray}  \label{fourier}
\label{fourier} \phi(\tau,z_{1},z_{2}) = \sum_{n = 0}^{\infty} \underset{\underset{nm \geq N(r)}{r \in \mathcal{O}_{K}^{\sharp}}}\sum c_{\Phi}(n , r) e^{ 2 \pi i  \left( n \tau + r z_{1} + \bar{r} z_{2} \right)}    
\end{eqnarray}
where $\mathcal{O}_{K}^{\sharp} = \frac{i}{2} \mathcal{O}_{K}$ \q (the inverse different of\q  $K|\mbb{Q}$) and $N \colon \mc{K} \rightarrow \mbb{Q}$ is the norm map.

$\phi$ is called a Jacobi \textit{cusp} form if $ c_{\phi}(n , r) = 0$ for $nm = N(r)$. The space of Jacobi cusp forms of weight $k$ and index $m$ is denoted $J_{k,m}^{cusp}(\mathcal{O}\sb{K} )$.   

Further, we denote the space of Jacobi forms of weight $k$ and index $m$ for the Jacobi group $SL(2,\mbb{Z}) \ltimes \mbb{Z}^{2}$ by $J_{k,m}$, elliptic modular(cusp) forms of weight $k$ for $SL(2,\mbb{Z})$ by $M_{k}$(resp. $S_{k}$).
We let $e(z) := e^{2 \pi i z}$ unless otherwise mentioned. In the rest of the paper we will use the standard notation $T := \left( \begin{smallmatrix}
1 & 1 \\
0 & 1  \end{smallmatrix}\right)$ and $S := \left( \begin{smallmatrix}
0 & -1 \\
1 & 0  \end{smallmatrix}\right)$.

\subsection{Theta Decomposition} 
Hermitian Jacobi forms admit a Theta decomposition analogous to classical jacobi forms. Let $\phi \in J_{k,m}(\mathcal{O}\sb{K} )$ has the Fourier expansion \ref{fourier}. Then we have the \texttt{theta decomposition}
\begin{align} \label{thetadecomposition}
\noindent \phi(\tau,z_{1},z_{2}) = \underset{s \in \mathcal{O}_{K}^{\sharp}/m \mathcal{O}_{K}}\sum h_{s}(\tau) \cdot \theta^{H}_{m,s}(\tau,z_{1},z_{2}), \q \mbox{ where } \nonumber \\
 h_{s}(\tau) := \underset{N(s)+L/4 \in m \mbb{Z}}{\sum_{L=0}^{\infty}} c_{s}(L) e^{ \frac{2 \pi i L \tau}{4m}} \mbox{ and } \theta^{H}_{m,s}(\tau,z_{1},z_{2}):= \underset{r \equiv s(mod \, m\mathcal{O}_{K})}\sum e \left(\frac{N(r)}{m}\tau + rz_{1}+\bar{r}z_{2} \right).  \nonumber \\
\end{align}

The Theta components of $h_{s}$ of $ \phi \in J_{k,m}(\mathcal{O}\sb{K} )$ (see \cite[p.46,47]{haverkamp}) have the following transformation properties under $SL(2,\mbb{Z})$ and $\mc{O}_{K}^{\times}$:
\begin{align}
& h_{s} \mid_{k-1} T = e^{ - 2\pi i N(s)/m} h_{s} \label{2.5}\\
& h_{s} \mid_{k-1} S = \frac{i}{4m} \underset{s^{\p} \in \mathcal{O}_{K}^{\sharp}/m \mathcal{O}_{K}}\sum e^{- 4 \pi i \mathrm{Re}(\bar{s} s^{\p}) /m} h_{s^{\p}} \label{2.6} \\
& h_{s} \mid_{k-1} \epsilon I = \epsilon h_{\epsilon s}, \qq  \epsilon \in \mc{O}_{K}^{\times}. \label{units}
\end{align}

\section{\x{Comparision of $J_{k,1}$ and $J_{k,1}(\mathcal{O}_{K})$}} \label{index1}
As a set of representatives of $\mathcal{O}_{K}^{\sharp}$ in $\mathcal{O}_{K}^{\sharp}/ \mathcal{O}_{K}$ ($\cong \frac{\mbb{Z}}{2 \mbb{Z}} \times \frac{\mbb{Z}}{2 \mbb{Z}}$) we take $ \mc{S}_{1} := \left \{ 0, \frac{i}{2}, \frac{1}{2}, \frac{1+i}{2} \right \}$. In this section we denote the corresponding Theta components by $h_{i,j}$ and the Hermitian Theta functions of index $1$ by $\theta^{H}_{i,j}$, where $\left \{ i,j\right \} \in \{ 0,1 \}$. We denote the Jacobi Theta functions of index $1$ by $\theta_{1,0}(\tau,z)$, $\theta_{1,1}(\tau,z)$. Further we let \[ \vartheta_{0}(\tau) = \underset{r \in \mbb{Z}}\sum e \left(r^{2}\tau \right), \qq  \vartheta_{1}(\tau) = \underset{r \equiv 1\pmod{2}}\sum  e\left(\frac{r^{2}}{4}\tau \right) \q (\tau \in \mc{H}) \]

\subsection{The case $k \equiv 2 \pmod{4}$}

\begin{thm} \label{2mod4}
\begin{enumerate}

\item

Let $k \equiv 2 \pmod{4}$. Then there is an exact sequence of vector spaces
\begin{equation} \label{2mod4eqn} 0 \longrightarrow J_{k,1}(\mathcal{O}_{K}) \overset{\pi_{1}}\longrightarrow J_{k,1} \overset{D_{0}}\longrightarrow M_{k} \longrightarrow 0 \end{equation}
where $D_{0}$ denotes the restriction to modular forms $\phi(\tau,z) \mapsto \phi(\tau,0)$.

\item
Let $k \equiv 2 \pmod{4}$. Then  $\pi_{1+i}$ is the zero map.
\end{enumerate}
\end{thm}

\begin{proof}
$1$. \, Let $\phi \in J_{k,1}(\mathcal{O}_{K})$. From the Theta decomposition~(\ref{thetadecomposition}) and the fact that when $k \equiv 2 \pmod{4}$, $h_{0,0} = h_{1,1} = 0$ and $h_{0,1} = - h_{1,0}$ (follows from equation~(\ref{units})) we get that \[ \pi_{1} \phi = h_{0,1} \left( \vartheta_{1} \theta_{1,0} - \vartheta_{0} \theta_{1,1} \right). \]
Since $ \vartheta_{1} \theta_{0} - \vartheta_{0} \theta_{1}  \not \equiv 0$ (see \cite{bocherer}), we clearly have that $\pi_{1}$ is injective and $\mathrm{Im}( \pi_{1}) \subseteq \ker{D_{0}}$.

Let $\phi \in \ker{D_{0}}$. From \cite[Theorem 1]{bocherer} we see that $\phi(\tau,z) = \varphi(\tau) \left( \vartheta_{1} \theta_{1,0} - \vartheta_{0} \theta_{1,1} \right)$, where $\varphi \in M_{k-1}(SL(2,\mbb{Z}),\bar{\omega})$ which consists of holomorphic functions $f \colon \mc{H} \rightarrow \mbb{C}$ bounded at infinity and satisfying  $f \mid_{k-1} S = \bar{\omega}(S) f, \, f \mid_{k-1} T = \bar{\omega}(T) f$ (also see \cite{bocherer}). $\omega$ is the linear character of $SL(2,\mbb{Z})$ defined by $\omega (T ) = i, \, \omega (S)  = i$.

So, we only need to check if $\phi \in J_{k,1}(\mathcal{O}_{K})$, then $h_{0,1} \in  M_{k-1}(SL(2,\mbb{Z}),\bar{\omega})$. We already know $h_{0,1}$ is in $M_{k-1}(\left( \Gamma(4)\right)$, so it suffices to check it has the right transformation properties under $SL(2,\mbb{Z})$. From~(\ref{2.5}),~(\ref{2.6})  we have \[  h_{0,1} \mid_{k-1} T  = e^{-2 \pi i N(i/2)} h_{0,1}, \q h_{0,1} \mid_{k-1} S  = \frac{ i}{2} \underset{s \in \mathcal{O}_{K}^{\sharp}/ \mathcal{O}_{K}}\sum e^{- 4 \pi i \mathrm{Re}(- is/2)} h_{s} \]
which give $h_{0,1} \mid_{k-1} T  = - i h_{0,1}(\tau)$ and $h_{0,1}\mid_{k-1} S  = - i h_{0,1}(\tau)$, as desired.

Finally $D_{0}$ is surjective. This is well known in view of the isomorphism $D_{0} + D_{2} \colon J_{k,1} \rightarrow M_{k} \oplus S_{k+2}$, where $D_{2}$ is defined in Corollary \ref{Sk2}. Another way to see this, using Hermitian Jacobi forms is as follows. Let $V := \mathrm{Im}(D_{0})$. Then we have the exact sequence \[ 0 \longrightarrow J_{k,1}(\mathcal{O}_{K}) \overset{\pi_{1}}\longrightarrow J_{k,1} \overset{D_{0}}\longrightarrow V \longrightarrow 0 \]
Therefore $\dim{V} = \dim{J_{k,1}} - \dim{J_{k,1}(\mc{O}_{K})} =  \dim{M_{k}}$. The last equality can be seen as follows. First let $k >4$. Then from \cite[p.25]{haverkamp} we easily compute when $k \equiv 2 \pmod{4}$ that \[ \dim{J_{k,1}^{Eis}(\mathcal{O}_{K})} =0 \q \mbox{ and from \cite[p.93]{haverkamp} we have } \q \dim{J_{k,1}^{cusp}(\mathcal{O}_{K})} = \left[ \frac{k+2}{12} \right] \] Since $J_{k,1}(\mathcal{O}_{K}) = J_{k,1}^{Eis}(\mathcal{O}_{K}) \oplus J_{k,1}^{cusp}(\mathcal{O}_{K})$ (see \cite{haverkamp}), we get the desired equality of dimensions. When $k = 2$, $J_{k,1}=0$ and hence so is it's subspace $J_{k,1}(\mathcal{O}_{K})$. This shows $V = M_{k}$ and completes the exactness of the sequence~(\ref{2mod4eqn}).  

$2$. \, Let $\phi \in J_{k,1}(\mathcal{O}_{K})$ have the Theta decomposition~(\ref{thetadecomposition}). From Lemma~\ref{Utheta}, we write down the Theta decomposition of $\pi_{1+i} \phi $:
\begin{align*}
\pi_{1+i} \phi &= (h_{0,0} a_{0} + h_{1,1} a_{2} ) \theta_{2,0} + (h_{1,0} a_{1} + h_{0,1} a_{3} ) \theta_{2,1} \\
& + (h_{0,0} a_{2} + h_{1,1} a_{0} ) \theta_{2,2} + (h_{0,1} a_{0} + h_{1,0} a_{3} ) \theta_{2,3},
\end{align*}
where $\theta_{2,\mu} := \theta_{2,\mu}(\tau,z), \, (\mu \in \mbb{Z} / 4 \mbb{Z} )$ are the Jacobi Theta functions of index $2$ and $a_{\mu} := \theta_{2,\mu}(\tau,0)$ (note that $a_{1} = a_{3}$). Since $h_{0,0}=h_{1,1}=0$ and $h_{0,1}+h_{1,0}=0$, when $k \equiv 2 \pmod{4}$ the Proposition follows.
\end{proof}

From the above Theorem and the results of \cite{bocherer} we get an isomorphism of $J_{k,1}(\mathcal{O}_{K})$ with $S_{k+2}$, which was also obtained by R. Sasaki in \cite{sasaki}.

\begin{cor}
Let $k \equiv 2 \pmod{4}$. Then $J_{k,1}(\mathcal{O}_{K}) \cong M_{k-1}(SL(2,\mbb{Z}),\bar{\omega})$.
\end{cor}

\begin{proof}
Let $\phi \in J_{k,1}(\mathcal{O}_{K})$. It follows from the proof of the above Theorem that the map sending $\phi$ to $h_{0,1}$ gives the desired isomorphism.
\end{proof}

\begin{cor} \label{Sk2}
Let $k \equiv 2 \pmod{4}$. Then the composite  
\begin{equation}
J_{k,1}(\mathcal{O}_{K}) \overset{\pi_{1}}\hookrightarrow J_{k,1} \overset{D_{2}}\rightarrow S_{k+2}
\end{equation}
gives an isomorphism from $J_{k,1}(\mathcal{O}_{K})$ with $S_{k+2}$. $ \left( D_{2} = \left( \frac{k}{2 \pi i} \frac{\partial^{2}}{\partial z^{2}} - 2 \frac{\partial}{\partial \tau} \right)_{z=0} \right)$  

\begin{proof}
The result follows from \cite[Theorem 2]{bocherer}, which in the case $N=1$ says that $ D_{2} \colon J_{k,1} \rightarrow S_{k+2}$ gives an isomorphism of $\ker{D_{0}}$ with  \[ S_{k+2}^{\circ}:= \left \{ f \in S_{k+2} \mid \varphi := \frac{f}{\xi} \in M_{k-1}(SL(2,\mbb{Z}), \bar{\omega}) \right \},\] where $\omega$ is defined as above and $\xi = \vartheta_{1}\vartheta^{\p}_{0} - \vartheta_{0}\vartheta^{\p}_{1}$. But $S_{k+2}^{\circ}= S_{k+2}$ when $N=1$, since by \cite[Proposition 2]{bocherer}, $\xi \in S_{3}(SL(2,\mbb{Z}), \omega)$. From equation~(\ref{2mod4eqn}) we have $\mathrm{Im} (\pi_{1}) = \ker{D_{0}}$. Therefore the Corollary follows.
\end{proof}
\end{cor}

\begin{cor} \label{multiplier}
Let $k \equiv 2 \pmod{4}$. Then multiplication by $\xi$ gives an isomorphism \\ $M_{k-1}(SL(2,\mbb{Z}),\bar{\omega}) \overset{\cong}\longrightarrow S_{k+2}$.
\end{cor}

\begin{proof}
Follows from the previous two Corollaries.
\end{proof}

We define $J_{k,1}(\mc{O}_{K},N)$ to be the space of Hermitian Jacobi forms for the congruence subgroup $\Gamma_{0}(N)$ in the usual way. It is immediate that the same proof as in Proposition \ref{2mod4} applies to this case when $k \equiv 2 \pmod{4}$ (see also \cite{bocherer} where the case of classical Jacobi forms is done) and we have an exact sequence of vector spaces \begin{equation} \label{2mod4levelN} 0 \longrightarrow J_{k,1}(\mc{O}_{K},N) \overset{\pi_{1}}\longrightarrow J_{k,1}(N) \overset{D_{0}}\longrightarrow M_{k}(N)  \end{equation} 

\begin{cor}
Let $N > 1$. Then $J_{2,1}(\mc{O}_{K},N) = 0$.
\end{cor}

\begin{proof}
A result of T. Arakawa, S. B$\ddot{\mbox{o}}$cherer \cite{bocherer2} says that $D_{0}$ in~(\ref{2mod4levelN}) is injective when $k=2$ and $N > 1$. Therefore the Corollary follows.
\end{proof}

\subsection{The case $k \equiv 0 \pmod{4}$}

We now treat the case $k \equiv 0 \pmod{4}$. We recall that 
\begin{equation} \label{jacobiiso} 
D_{0}+D_{2} \colon J_{k,1} \overset{\cong}\longrightarrow M_{k} + S_{k+2}; \qq D_{2} = \frac{2k}{(2 \pi i)^{2}} \psi_{2} - \frac{2}{2 \pi i} \psi_{0}^{\p} \, \colon J_{k,1} \rightarrow S_{k+2}
\end{equation}
where, $\phi(\tau,z)= \underset{\nu \geq 0}\sum \psi_{\nu}(\tau) z^{\nu} \in J_{k,1}$ is the Taylor expansion of $\phi$ around 
$z=0$.

Further from the Taylor expansion of Hermitian Jacobi forms, one can define the $D_{\nu}(\mc{O}_{K})$ operators in the same way as for the case of Jacobi forms (see \cite{somu}, \cite{sasaki}). Let $ \phi(\tau,z_{1},z_{2})= \underset{\alpha,\beta \geq 0}\sum \chi_{\alpha,\beta}(\tau) z_{1}^{\alpha} z_{2}^{\beta} \in J_{k,1}(\mc{O}_{K})$ be the Taylor expansion of $\phi$ around $z_{1}=z_{2}=0$. Let $\xi_{1,1}:= D_{1}(\mc{O}_{K})\phi, \, \xi_{2,2}:= D_{2}(\mc{O}_{K})\phi$. Then 
\begin{equation}\label{Dhermitian}
\xi_{1,1} := \chi_{1,1} - \frac{2 \pi i}{k} \chi_{0,0}^{\p}, \q \xi_{2,2} := \chi_{2,2}- \frac{2 \pi i}{k+2} \chi_{1,1}^{\p} + \frac{(2 \pi i)^{2}}{2(k+1)(k+2)} \chi_{0,0}^{\p \p}
\end{equation}
define linear maps from $J_{k,1}(\mc{O}_{K})$ to $ S_{k+2}$ and from $J_{k,1}(\mc{O}_{K})$ to $ S_{k+4}$ respectively.

R. Sasaki, in \cite{sasaki} proved that when $k \equiv 0 \pmod{4}$
\begin{equation} \label{sasakikmod4}
\xi \colon J_{k,1}(\mc{O}_{K}) \rightarrow M_{k} \oplus S_{k+2} \oplus S_{k+4}, \qq \phi \mapsto \chi_{0,0}+\xi_{1,1}+\xi_{2,2}-6(\chi_{4,0}+\chi_{0,4}) 
\end{equation}
is an isomorphism.

\begin{rem}
We remark here that from the Fourier expansion of a Hermitian Jacobi form $\phi$~(\ref{fourier}) of index $1$, we get $\phi(\tau,z_{1},z_{2}) = \phi(\tau,z_{2},z_{1})$ if $k \equiv 0 \pmod{4}$ and hence in it's Taylor expansion we have $\chi_{\alpha,\beta} = \chi_{\beta,\alpha}\, \forall \, \alpha,\beta \geq 0$. Hence the isomorphism is also given by $ \phi \mapsto \chi_{0,0}+\xi_{1,1}+\xi_{2,2}-12(\chi_{0,4})$. Hence the $4$ Taylor coefficients $ \chi_{0,0},\chi_{0,4},\chi_{1,1},\chi_{2,2}$ determine $\phi$, as expected in analogy with classical Jacobi forms.
\end{rem}

\begin{thm}\label{0mod4}
\begin{enumerate}

\item

Let $k \equiv 0 \pmod{4}$. Then there is an exact sequence of vector spaces
\begin{equation} \label{0mod4} 0 \longrightarrow  S_{k+4} \overset{ \xi^{-1}\mid_{S_{k+4}} }\longrightarrow J_{k,1}(\mathcal{O}_{K}) \overset{\pi_{1}}\longrightarrow J_{k,1} \longrightarrow 0 \end{equation}
where $\xi \colon J_{k,1}(\mathcal{O}_{K}) \rightarrow  M_{k}\oplus S_{k+2} \oplus S_{k+4} $ is the isomorphism given in \cite{sasaki}.

\item

Let $k \equiv 0 \pmod{4}$. Then $\pi_{1+i}$ induces an isomorphism between $J_{k,1}(\mathcal{O}_{K})$ and  $J_{k,2}$. 
\end{enumerate}
\end{thm} 

\begin{proof}
$1$. \, Follows directly from Lemma~\ref{comm} given below.

$2$. \, When $k \equiv 0 \pmod{4}$, in the Theta decomposition of $\phi \in J_{k,1}(\mathcal{O}_{K})$ we have $h_{0,1}=h_{1,0}$. Let $\phi \in \ker{\pi_{1+i}}$. From the Theta decomposition of $\pi_{1+i}\phi$ (see \cite{somu2}) we easily deduce $h_{0,1}=h_{1,0} =0, \, (a_{0}^{2} - a_{2}^{2}) h_{0,0}=  (a_{0}^{2} - a_{2}^{2})  h_{1,1} = 0$. But $a_{0}^{2} \not \equiv a_{2}^{2}$ since $Wr_{2}$ doesnot vanish on $\mc{H}$ (see the proof of \textit{Step 1} of Theorem~\ref{0mod4ind2}). Hence, the kernel is trivial. Moreover, from Corollary~\ref{3.7}, considering the domensions, we conclude that $\pi_{1+i}$ is an isomorphism.
\end{proof}

\begin{cor}\label{3.7}
Let $k \equiv 0 \pmod{4}$. Then \[ J_{k,2} \overset{D_{0}+D_{2}+D_{4}}\longrightarrow M_{k}\oplus S_{k+2} \oplus S_{k+4} \overset{\xi^{-1}}\longrightarrow J_{k,1}(\mathcal{O}_{K}) \]
is an isomorphism.
\end{cor}

\begin{proof}
In fact, each map is an isomorphism. The first map is injective by \cite{zagier} and dimension count shows it is an isomorphism.
\end{proof}

\begin{rem}
In the above Theorem, it is clear that if $f \in S_{k+4}$,
\[ \xi^{-1}f = \left\{ \phi \in J_{k,1}(\mathcal{O}_{K}) \mid \chi_{2,2}-12 \, \chi_{0,4} =f   \right\} \]
\end{rem}

\begin{lem}\label{comm}
The following diagram
\[ \begin{CD} 
J_{k,1}(\mc{O}_{K})                 @>\pi_{1} >>  J_{k,1}\\
@V{\cong}V{\xi}V                   @V{\cong}V{D_{0}+\frac{(2 \pi i)^{2}}{2 k}D_{2}}V\\
M_{k}\oplus S_{k+2} \oplus S_{k+4}  @>pr.>>  M_{k} \oplus S_{k+2} 
\end{CD} \]
is commutative.
\end{lem}

\begin{proof}
The proof is immediate from definitions. We compute $(pr. \circ \xi)\phi = \chi_{0,0}- \frac{2 \pi i}{k} \chi_{0,0}^{\p} + \chi_{1,1}$. On the other hand, $(D_{0}+\frac{(2 \pi i)^{2}}{2 k}D_{2}) \circ \pi_{1} \phi = \chi_{0,0} + (\chi_{0,2}+\chi_{2,0}+\chi_{1,1}) - \frac{2 \pi i}{k} \chi_{0,0}^{\p} = \chi_{0,0} - \frac{2 \pi i}{k} \chi_{0,0}^{\p} + \chi_{1,1}$, since $\chi_{0,2}= \chi_{2,0}= 0 = \chi_{1,0} = \chi_{0,1}$ when $k \equiv 0 \pmod{4}$. In fact, $\chi_{\alpha,\beta}=0$ unless $\alpha-\beta \equiv k \pmod{4}$ follows from first transformation rule for Hermitian Jacobi forms~(\ref{jacobi1}).
\end{proof}

\section{Hermitian Jacobi forms of index 2} \label{index2}
In this section we consider Hermitian Jacobi forms of index $2$ by relating them to classical Jacobi forms and elliptic modular forms via several restriction maps. Let $\mc{D}:= 2 i \mc{O}_{K}$, the Different of $K$. We use a representation of the group defined for a positive integer $m$ : \[ G_{m}:= \left\{ \mu \in \mc{O}_{K}/m \mc{D} \mid N(\mu) \equiv 1 \pmod{4m} \right\} .\] For $m = 2$, we consider the representation of $G_{2} = U_{K} \cong \mbb{Z}/4 \mbb{Z}$ defined in \cite{haverkamp} :
\[ \rho_{2} \colon G_{2} \longrightarrow Aut \left( J_{k,2}(\mathcal{O}_{K}) \right), \q \mu \mapsto W_{\mu} ,\]
where $W_{\mu}$ is defined by  \[ W_{\mu} \left( \Theta_{2}^{t} \cdot h \right) = \Theta_{2}^{t} \cdot h^{(\mu)}, \q h^{(\mu)} :=\left( h_{\mu s} \right)_{s \in \mathcal{O}_{K}^{\sharp}/2 \mathcal{O}_{K} } .\]

Accordingly we have a decomposition  of $J_{k,2}(\mathcal{O}_{K})$ : \begin{equation} J_{k,2}(\mathcal{O}_{K}) = \underset{ \eta \in G_{2}^{*}}\bigoplus  J_{k,2}^{\eta}(\mathcal{O}_{K}), \label{chardecom} \end{equation} 
where, $G_{2}^{*}$ is the group of characters of $G$ and \begin{equation} J_{k,2}^{\eta}(\mathcal{O}_{K}):= \left\{ \phi \in J_{k,2}(\mathcal{O}_{K}) \mid W_{\mu} \phi = \eta(\mu) \phi \q \forall \mu \in G_{2}  \right\} .\end{equation}

$ G_{2} \cong \mbb{Z}/4 \mbb{Z} \mbox{ via } \, i \mapsto 1, \, -1 \mapsto 2, \, -i \mapsto 3, \, 1 \mapsto 0$. Also, \[ G_{2}^{*} = \left\{ \eta_{\alpha}:= \left( x \mapsto e^{\frac{2 \pi i \alpha x}{4} } \right); \, x, \alpha \in \mbb{Z}/4 \mbb{Z} \right\} .\]

We take as a set of representatives of $\mc{O}^{\sharp}_{K}$ in $\mathcal{O}_{K}^{\sharp}/m \mathcal{O}_{K}$ as the set \[ \mc{S}_{m}:= \left\{ \frac{a}{2}+i \frac{b}{2} \, \mid \, a,b \in \mbb{Z}/2 m \mbb{Z} \right\}. \] We denote the corresponding Theta components of $\phi \in J_{k,2}(\mathcal{O}_{K})$ by $h_{a,b}$ and the Hermitian Theta functions of weight $1$ and index $m$ by $\theta^{H}_{m;a,b}$ (or by $\theta^{H}_{m;s}, \, s = \frac{a+ib}{2} \in \mc{S}_{m}$) in this section, but we drop the index unless there is a danger of confusion. Also we denote by $\theta_{m,\mu}(\tau,z)$ ($\mu \pmod{2m}$) the classical Theta functions.  
The following Lemmas give the Theta decomposition of the images of Hermitian Jacobi forms of index $2$ under the restriction maps. We define for convenience of notation $a_{\mu} := \theta_{2,\mu}(\tau,0)$ ($ \mu \in \mbb{Z}/4 \mbb{Z}$) and $b_{\mu} := \theta_{4,\mu}(\tau,0)$ ($ \mu \in \mbb{Z}/8 \mbb{Z}$).

\begin{lem} \label{2to2}
Let $ \pi_{1} \colon J_{k,2}(\mathcal{O}_{K}) \rightarrow J_{k,2}$ given by $\phi(\tau,  z_{1},  z_{2})\mapsto \phi(\tau,  z , z) $ and $ \pi_{1} \phi = \underset{\mu \in \mbb{Z}/ 4 \mbb{Z}}\sum H_{\mu}(\tau) \cdot \theta_{2,\mu}(\tau,z)$ be it's Theta decomposition, where $H_{\mu} = (-1)^{k} H_{- \mu}$ ($\mu \in \mbb{Z} / 4 \mbb{Z} $). Then, 
\begin{align}
& H_{0} = h_{0,0} a_{0}  + h_{0,1} a_{1}  + h_{0,2} a_{2}  + h_{0,3} a_{3},  \\
& H_{1} = h_{1,0} a_{0}  + h_{1,1} a_{1}  + h_{1,2} a_{2}  + h_{1,3} a_{3},  \\
& H_{2} = h_{2,0} a_{0}  + h_{2,1} a_{1}  + h_{2,2} a_{2}  + h_{2,3} a_{3} . 
\end{align}
\end{lem}

\begin{proof}
Let $s \in \mc{S}_{2}$. The effect of $\pi_{1}$ on $\theta^{H}_{2;s}$ is given below.
\begin{align*} 
\pi_{1} \theta^{H}_{2;s}  &= \underset{ r \in \mc{O}^{\sharp}_{K}}{\underset{r \equiv s \pmod{2 \mc{O}_{K}}}\sum} e \left( \frac{N(r)}{2} \tau + 2 \mathrm{Re}(r) \cdot z \right)\\
&= \underset{ \mrm{Re}(2 r) \in \mbb{Z}}{\underset{\mrm{Re}(2 r) \equiv \mrm{Re}(2 s) \pmod{4 \mbb{Z}}}\sum} e \left( \frac{(\mrm{Re}(2 r))^{2}}{8} \tau +  \mrm{Re}(2 r) \cdot z \right) \times \underset{ \mrm{Im}(2 r) \in \mbb{Z}}{\underset{\mrm{Im}(2 r) \equiv \mrm{Im}(2 s) \pmod{4 \mbb{Z}}}\sum} e \left( \frac{(\mrm{Im}(2 r))^{2}}{8} \tau \right) \\
&= \theta_{2, \mathrm{Re}(2s)}( \tau,z) \cdot a_{\mathrm{Im}(2s)}. 
\end{align*}
This shows that $ \pi_{1} \phi = \underset{\mu \in \mbb{Z} / 4 \mbb{Z}}\sum \left( \underset{\mathrm{Re}(2s) = \mu}{\underset{s \in \mc{S}_{2}}\sum} h_{s} \cdot a_{\mathrm{Im}(2s)} \right) \theta_{2,\mu}(\tau,z) $, which proves the Lemma.
\end{proof}

\begin{lem} \label{2to4}
Let $\pi_{1+i} \colon J_{k,2}(\mathcal{O}_{K}) \rightarrow J_{k,4}$ given by $\phi(\tau,  z_{1},  z_{2})\mapsto \phi(\tau, (1+i) z , (1-i) z) $ and $\pi_{1+i} \phi = \underset{\mu \in \mbb{Z}/ 8 \mbb{Z}}\sum \bar{H}_{\mu}(\tau) \cdot \theta_{4,\mu}(\tau,z)$ be it's Theta decomposition, where $\bar{H}_{\mu} = (-1)^{k} \bar{H}_{ - \mu}$ ($\mu \in \mbb{Z} / 8 \mbb{Z} $). Then, 
\begin{align}
& \bar{H}_{0} = h_{0,0} b_{0}  + h_{1,1} b_{2}  + h_{2,2} b_{4}  + h_{3,3} b_{6},  \\
& \bar{H}_{1} = h_{1,0} b_{1}  + h_{2,1} b_{3}  + h_{3,2} b_{5}  + h_{0,3} b_{7},  \\
& \bar{H}_{2} = h_{2,0} b_{2}  + h_{3,1} b_{4}  + h_{0,2} b_{6}  + h_{1,3} b_{0}, \\
& \bar{H}_{3} = h_{3,0} b_{3}  + h_{0,1} b_{5}  + h_{1,2} b_{7}  + h_{2,3} b_{1}, \\
& \bar{H}_{4} = h_{0,0} b_{4}  + h_{1,1} b_{6}  + h_{2,2} b_{0}  + h_{3,3} b_{2}.
\end{align}
\end{lem}

\begin{proof}
We note that $ 2(1+i) \mc{O}_{K} = 4  \mc{O}_{K} \cup 2(1+i) + 4 \mc{O}_{K} $ (disjoint union) as abelian groups. Let $s = \frac{\mu}{2} + i \frac{\lambda}{2} \in \mc{S}_{2}$. We have $U_{1+i} \theta_{2, s}^{H} (\tau , z_{1} , z_{2}) = $
{\scriptsize \begin{align*}
&= \underset{r \equiv s \pmod{2 \mc{O}_{K}} }\sum e \left( \frac{N(r)}{2}\tau + (1+i)r z_{1} + (1-i)\bar{r} z_{2} \right) \\
&= \underset{r^{\p} \equiv (1+i) s \pmod{2(1+i)\mc{O}_{K}}}\sum e \left( \frac{N(r^{\p})}{4}\tau + r^{\p} z_{1} + \bar{r}^{\p} z_{2} \right) \\
&=  \underset{r^{\p} \equiv \frac{\mu - \lambda}{2} + i \frac{\mu + \lambda}{2} \pmod{4 \mc{O}_{K}}}\sum  e \left( \frac{N(r^{\p})}{4}\tau + r^{\p} z_{1} + \bar{r}^{\p} z_{2} \right) + \underset{r^{\p} \equiv \frac{\mu - \lambda + 4}{2} + i \frac{\mu + \lambda + 4}{2} \pmod{4 \mc{O}_{K}}}\sum e \left( \frac{N(r^{\p})}{4}\tau + r^{\p} z_{1} + \bar{r}^{\p} z_{2} \right),
\end{align*} }
from which the Lemma follows easily.
\end{proof}

From the transformation $h_{s}\mid_{k-1} \epsilon I = \epsilon h_{\epsilon s}$ ($\epsilon \in \mc{O}_{K}^{\times}$), we conclude that 
\begin{equation}\label{4.3} h_{a,b}= i^{k}h_{-b,a},\q  h_{a,b}= (-1)^{k}h_{-a,-b}. \end{equation}
From the direct-sum decomposition~(\ref{chardecom}) or from above equations~(\ref{4.3}) we see that  
$J_{k,2}(\mathcal{O}_{K}) = J_{k,2}^{\eta_{\alpha}}(\mathcal{O}_{K}) $ for $k + \alpha \equiv 0 \pmod{4}$. 

\subsection{$ \eta = \eta_{1}$} \label{4.1} In this case $ k \equiv 3 \pmod{4}$, and it is easy to see that $ h_{0,0} = h_{2,2} = h_{0,2} = h_{2,0} = 0$, and after a calculation, 
\begin{align} 
& h_{0,3}= -h_{0,1},\q h_{1,0}= -i h_{0,1},\q h_{1,3}= -i h_{1,1},\q h_{2,1}= i h_{1,2},\q h_{2,3}= -i h_{1,2} \label{3mod41}\\
& h_{3,0}= i h_{0,1},\q h_{3,1}= i h_{1,1},\q h_{3,2}= - h_{1,2},\q h_{3,3}= -h_{1,1} \label{3mod42} .
\end{align}

We consider the map $\pi_{1+i} \colon J_{k,2}^{\eta_{1}}(\mathcal{O}_{K}) \rightarrow J_{k,4}$. Using Lemma~\ref{2to4} we have
\begin{align}
& \pi_{1+i} \phi(\tau,z) = \underset{\mu \pmod{8}}\sum \bar{H}_{\mu} \theta_{4,\mu}(\tau,z) \q \mbox{ where } \bar{H}_{0}= \bar{H}_{4}= 0 \\
& \bar{H}_{1}= -(1+i) h_{0,1} b_{1}  - (1-i) h_{1,2} b_{3}, \bar{H}_{2} = i h_{1,1} (b_{4} - b_{0}), \\
& \bar{H}_{3}= (1+i) h_{0,1} b_{3}  + (1-i) h_{1,2} b_{1}
\end{align}
from which we conclude that $\pi_{1+i}$ is injective. But for $k > 4$, from \cite[Satz 2.5]{haverkamp} we get $\dim{J_{k,2}^{Eis}(\mathcal{O}_{K})  }=0$. Also for $k > 4$, using the Trace formula (see \cite[Theorem 3]{haverkamp/en}, \cite[Korollar 2.5, p.92]{haverkamp}) we get $\dim{J_{k,2}(\mathcal{O}_{K})} = \frac{k-3}{4} = \dim{J_{k,2}}$, where the last equality follows from \cite[Cor. Theorem 9.2]{zagier}). When $k = 3$, $J_{3,4} = 0$ and therefore so is $J_{3,2}(\mathcal{O}_{K}) $. Therefore,

\begin{prop} \label{3mod4ind2}
Let $k \equiv 3 \pmod{4}$. Then $\pi_{1+i}$ induces an isomorphism between $J_{k,2}(\mathcal{O}_{K}) $ and $J_{k,4}$.
\end{prop}

\subsection{$ \eta = \eta_{2}$}  \label{4.2} In this case $k \equiv 2 \pmod{4}$ and using the equations~(\ref{units}) and~(\ref{4.5}) we find that $h_{0,0} = h_{2,2} =0$ and every other Theta component $h_{s}$ of $\phi \in J_{k,2}^{\eta_{2}}(\mathcal{O}_{K})$ is an unit times $h_{0,1}, h_{0,2}, h_{1,1}, h_{1,2}$ :
\begin{align}
& h_{0,3}= h_{0,1},\q h_{1,0}= - h_{0,1},\q h_{1,3}= - h_{1,1},\q h_{2,1}= - h_{1,2},\q h_{2,3}= - h_{1,2} \label{2mod41}\\
& h_{3,0}= - h_{0,1},\q h_{3,1}= - h_{1,1},\q h_{3,2}=  h_{1,2},\q h_{3,3}= h_{1,1} .\label{2mod42} 
\end{align}

Further, we calculate the transformation of $h_{0,1}, h_{0,2}, h_{1,1}, h_{1,2}$ under $S$ from equation~(\ref{2.6}):
\begin{align}
&h_{0,1} \mid_{k-1} S  = \frac{i}{2} (h_{0,1} + h_{0,2} + h_{1,2} ) \label{h01} \\
&h_{0,2} \mid_{k-1} S  = i (h_{0,1} - h_{1,2} ) \label{h02} \\
&h_{1,1} \mid_{k-1} S  = -i h_{1,1} \label{h11} \\
&h_{1,2} \mid_{k-1} S  = \frac{i}{2}(  h_{0,1} -  h_{0,2} + h_{1,2} ) \label{h12}
\end{align}

Also, the formula $ h_{s} \mid_{k-1} T = e^{- \pi i N(s)} h_{s}$ (from~(\ref{2.5}) when $m =2$), gives 
\begin{align}
& h_{1,1}(\tau + 1) =  -i h_{1,1}(\tau), \q h_{0,1}(\tau + 1) =  \frac{1-i}{\sqrt{2}} h_{0,1}(\tau)\\
& h_{0,2}(\tau + 1) =  - h_{0,2}(\tau), \q h_{1,2}(\tau + 1) = -\frac{1-i}{\sqrt{2}}h_{1,2}(\tau)
\end{align} 

We note the above observations in the following Lemmas :
\begin{lem} \label{h11}
Let $ k \equiv 2 \pmod{4}$. Then in the Theta decomposition~(\ref{thetadecomposition}) of $\phi \in J_{k,2}(\mathcal{O}_{K})$, $h_{1,1} \in M_{k-1} \left( SL(2,\mbb{Z}), \bar{\omega} \right)$, where $\omega$ is the linear character of $SL(2,\mbb{Z})$ defined by $\omega \left( T \right) = \omega \left( S \right) = i$.
\end{lem}

\begin{proof}
From the above we get $h_{1,1}  \in M_{k-1} \left( SL(2,\mbb{Z}), \bar{\omega} \right)$; since $h_{1,1}$ is already a modular form for $\Gamma(8)$, the holomorphicity at infinity is automatic.
\end{proof}

\begin{lem}
$ k \equiv 2 \pmod{4}$. Then $J_{k,2}^{Spez}(\mathcal{O}_{K}) = 0$.
\end{lem}

\begin{proof}
For by \cite[Proposition 5.6]{haverkamp} the homomorphism $\iota \colon J_{k,2}^{Spez}(\mathcal{O}_{K}) \rightarrow M_{k-1}^{*} \left( \Gamma_{0}(8), ( \frac{-4}{\cdot} ) \right) $ is injective, where 
\begin{align} 
M_{k-1}^{*} \left( \Gamma_{0}(8), \left( \frac{-4}{\cdot} \right) \right) = \left\{ f(\tau) = \underset{n}\sum a(n) e(n \tau) \in M_{k-1} \left( \Gamma_{0}(8), \left( \frac{-4}{\cdot} \right) \right) \mid \right.\\
\left. a(n) \neq 0 \Rightarrow \exists \lambda \in \mc{O}_{K} \colon n \equiv - N(\lambda) \pmod{8}    \right\} 
\end{align}
and $\iota(\phi)(\tau) = \underset{s \in \in \mathcal{O}_{K}^{\sharp}/ \mathcal{O}_{K}}\sum h_{s}(8 \tau)$, which turns out to be $0$ in this case from the equations~(\ref{2mod41}) and~(\ref{2mod42}).
\end{proof}

\begin{lem} \label{chi60}
Let $p,q \in \mbb{Z} / 4 \mbb{Z}$ so that $\frac{p}{2} + i \frac{q}{2} \in \mc{S}_{2}$. Then
\begin{align*}
\frac{\partial^{6}}{\partial z_{1}^{6}} \left( \theta^{H}_{p,q}(\tau,z_{1},z_{2}) - \theta^{H}_{q,p}(\tau,z_{1},z_{2}) \right)_{z_{1}=z_{2}=0} = 2 \left( 16 \pi i \right)^{3} \Big( \left( a^{\p \p \p}_{p} a_{q} - a^{\p \p \p}_{q} a_{p}  \right) + 15 \left( a^{\p \p}_{q} a^{\p}_{p} - a^{\p \p}_{p} a^{\p}_{q}  \right) \Big)
\end{align*}
\end{lem}

\begin{proof}   
\begin{align*}
L.H.S. &=  \left( 2 \pi i \right)^{6} \underset{y \equiv q \pmod{4}}{\underset{x \equiv p \pmod{4}}\sum} (x + iy)^{6} e \left( \frac{x^{2}+y^{2}}{8}\tau \right) - \left( 2 \pi i \right)^{6} \underset{x \equiv q \pmod{4}}{\underset{y \equiv p \pmod{4}}\sum} (x + iy)^{6} e \left( \frac{x^{2}+y^{2}}{8}\tau \right)  \\
& = \left( 2 \pi i \right)^{6} \underset{y \equiv q \pmod{4}}{\underset{x \equiv p \pmod{4}}\sum} (x + iy)^{6} e \left( \frac{x^{2}+y^{2}}{8}\tau \right) + \left( 2 \pi i \right)^{6} \underset{y \equiv q \pmod{4}}{\underset{x \equiv p \pmod{4}}\sum} (x - iy)^{6} e \left( \frac{x^{2}+y^{2}}{8}\tau \right)  \\
& = 2 \left( 2 \pi i \right)^{6} \underset{y \equiv q \pmod{4}}{\underset{x \equiv p \pmod{4}}\sum} (x^{6} - 15 x^{4}y^{2} +15 x^{2}y^{4} - y^{6}) e \left( \frac{x^{2}+y^{2}}{8}\tau \right) \\
& = 2 \left( 16 \pi i \right)^{3} \Big( \left( a^{\p \p \p}_{p} a_{q} - a^{\p \p \p}_{q} a_{p}  \right) + 15 \left( a^{\p \p}_{q} a^{\p}_{p} - a^{\p \p}_{p} a^{\p}_{q}  \right) \Big) = R.H.S.
\end{align*}
\end{proof}

\begin{thm} \label{2mod4ind2}
Let $k \equiv 2 \pmod{4}$. We have the following exact sequence of vector spaces
\begin{align}
0 \longrightarrow S_{k+2} \times S_{k+6} \overset{\sigma}\longrightarrow J_{k,2}(\mathcal{O}_{K}) \overset{ \pi_{1}}\longrightarrow J_{k,2} \overset{D_{0}}\longrightarrow M_{k} \longrightarrow 0
\end{align}
\end{thm}
The map $\sigma$ is defined as follows. We will prove that $\ker{\pi_{1}} \cong M_{k-1}(SL(2,\mbb{Z}),\bar{\omega}) \times S_{k+6}$ ; $\phi \mapsto \Big(  h_{1,1}, \, D_{0}(6) (\phi - h_{1,1} \left( \theta^{H}_{1,1} - \theta^{H}_{1,3} - \theta^{H}_{3,1} + \theta^{H}_{3,3} \right) ) \Big)$ where $D_{0}(6) \phi = \chi_{6,0}$, the coefficient of $z_{1}^{6}$ in the Taylor expansion of $\phi$ around $z_{1} = z_{2} = 0$ (see Chapter $2$ for the definition of Differential operators $D_{\nu}, \, \nu \in \mbb{Z}_{\geq 0}$). $\sigma$ will be the inverse of this isomorphism composed with the isomorphism from $S_{k+2}$ to $M_{k-1}(SL(2,\mbb{Z}),\bar{\omega})$ (see Corollary~\ref{multiplier}).  

\begin{proof} We divide the proof into $3$ steps. 

\textit{Step 1.} Consider the restriction map $\pi_{1} \colon J_{k,2}(\mathcal{O}_{K}) \rightarrow J_{k,2}$. Let $\phi \in \ker{\pi_{1}}$. We obtain the Theta decomposition of $\pi_{1} \phi$ from Lemma~\ref{2to2}. Keeping the notation of the Lemma,
 \begin{align}
&H_{0} = h_{0,0} a_{0} + h_{0,1}a_{1}  + h_{0,2}a_{2}  + h_{0,3}a_{3}= 2h_{0,1}a_{1} + h_{0,2}a_{2} \label{21} \\ 
&H_{1} = h_{1,0} a_{0} + h_{1,1}a_{1}  + h_{1,2}a_{2}  + h_{1,3}a_{3}= - h_{0,1}a_{0} + h_{1,2}a_{2} \label{22} \\
&H_{2} = h_{2,0} a_{0} + h_{2,1}a_{1}  + h_{2,2}a_{2}  + h_{2,3}a_{3}= - h_{0,2}a_{0} - 2h_{1,2}a_{1} \label{23}
\end{align}
upon using equations~(\ref{2mod41}),~(\ref{2mod42}); and $H_{1} = H_{3}$. Since $\phi \in\ker{\pi_{1}}$ we get \begin{equation} \frac{h_{0,1}}{a_{2}} = \frac{- h_{0,2}}{2a_{1}} =  \frac{h_{1,2}}{a_{0}}: = \psi.  \label{psidefn} \end{equation} 
$\psi$ is well defined since it is well known that $a_{\mu}$ ($\mu \in \mbb{Z}/4 \mbb{Z}$) never vanish on $\mc{H}$. Therefore
\begin{align} \label{psi}
\phi = \, & \psi  \Big( a_{2} \left( \theta^{H}_{0,1} + \theta^{H}_{0,3} - \theta^{H}_{1,0} - \theta^{H}_{3,0} \right)  -2 a_{1} \left( \theta^{H}_{0,2} - \theta^{H}_{2,0} \right) + a_{0} \left( \theta^{H}_{1,2} - \theta^{H}_{2,1} - \theta^{H}_{2,3} + \theta^{H}_{3,2} \right)  \Big)\\
& + h_{1,1} \left( \theta^{H}_{1,1} - \theta^{H}_{1,3} - \theta^{H}_{3,1} + \theta^{H}_{3,3} \right). \nonumber
\end{align}
Furthermore from the definition of $\psi$ above and using the transformation formulas (\ref{h01}), (\ref{h02}), (\ref{h12}) for $( h_{0,1}, h_{0,2}, h_{1,2}) $ we get the following transformation formulas for $\psi$ :
\begin{align} \label{transfpsi}
\psi \left( - \frac{1}{\tau} \right) = \frac{1-i}{\sqrt{2}} \tau^{k - 3/2} \psi(\tau), \q \psi(\tau + 1) = - \frac{1-i}{\sqrt{2}}  \psi(\tau)
\end{align}

Further, from equations~(\ref{h01}),~(\ref{h02}),~(\ref{h12}),~(\ref{h11}) and from Theorem~\ref{theta} (since the transformations of $( h_{0,1}, h_{0,2}, h_{1,2}) $ and $h_{1,1}$ under $S,T$ are independent of each other) we conclude that 
$h_{1,1} \left( \theta^{H}_{1,1} - \theta^{H}_{1,3} - \theta^{H}_{3,1} + \theta^{H}_{3,3} \right) \in J_{k,2}(\mathcal{O}_{K})$ and hence so is $\phi - h_{1,1} \left( \theta^{H}_{1,1} - \theta^{H}_{1,3} - \theta^{H}_{3,1} + \theta^{H}_{3,3} \right)$.

We define $\ker{ \pi_{1}}^{\circ} := \left\{ \phi \in \ker{ \pi_{1}} \mid h_{1,1} = 0  \right\} $. By the same reasoning as in the above paragraph, 
\[ \ker{ \pi_{1}} \cong M_{k-1}(SL(2,\mbb{Z}),\bar{\omega}) \times \ker{ \pi_{1}}^{\circ} \q \mbox{ via } \q \phi \mapsto \left( h_{1,1}, \, \phi - h_{1,1} (\theta^{H}_{1,1} - \theta^{H}_{1,3} - \theta^{H}_{3,1} + \theta^{H}_{3,3}) \right),\] 
using Lemma~\ref{h11} and that $\theta^{H}_{1,1} - \theta^{H}_{1,3} - \theta^{H}_{3,1} + \theta^{H}_{3,3} \not \equiv 0$. The latter fact follows from (cf. \cite{haverkamp}) \[ \int_{P_{\tau}} \theta^{H}_{m;s} (\tau, z_{1}, z_{2}) \cdot \overline{\theta^{H}_{m;t} (\tau, z_{1}, z_{2})} e^{- \pi m N(z_{1} - \bar{z_{2}})/ v } dz_{1} dz_{2} = \frac{\delta_{s,t} (m \mc{O}_{K}) v}{m} \]
where, $\tau = u+ iv \in \mc{H}$, $\delta_{s,t}(m \mc{O}_{K}) := \begin{cases} 1 & \mbox{ if } s \equiv t \pmod{m \mc{O}_{K}} \\ 
0 & \mbox{ otherwise } \end{cases}$, and the parallelotope \\ $P_{\tau} := \left\{( \alpha + \beta i + \gamma \tau + \delta i \tau ), ( \alpha - \beta i + \gamma \tau - \delta i \tau ); \, 0 \leq \alpha, \beta, \gamma, \delta \leq 1   \right\} \subset \mbb{C}^{2}$. We now prove that $D_{0}(6) \colon \ker{ \pi_{1}}^{\circ} \rightarrow S_{k+6} $ is an isomorphism.

Let $\phi \in \ker{ \pi_{1}}^{\circ}$. From equation~(\ref{psi}) and Lemma~\ref{chi60} we get 
\begin{align*}
D_{0}(6) \phi &= 2 c \psi a_{2}\Big( \left( a^{\p \p \p}_{0} a_{1} - a^{\p \p \p}_{1} a_{0}  \right) + 15 \left( a^{\p \p}_{1} a^{\p}_{0} - a^{\p \p}_{0} a^{\p}_{1}  \right) \Big) - 2 c \psi a_{1}\Big( \left( a^{\p \p \p}_{0} a_{2} - a^{\p \p \p}_{2} a_{0}  \right) + 15 \left( a^{\p \p}_{2} a^{\p}_{0} - a^{\p \p}_{0} a^{\p}_{2}  \right) \Big) \\
& + 2 c \psi a_{0}\Big( \left( a^{\p \p \p}_{1} a_{2} - a^{\p \p \p}_{2} a_{1}  \right) + 15 \left( a^{\p \p}_{2} a^{\p}_{1} - a^{\p \p}_{1} a^{\p}_{2}  \right) \Big) \\
&= 30 c \psi \Big( a_{0} \left( a^{\p \p}_{2} a^{\p}_{1} - a^{\p \p}_{1} a^{\p}_{2} \right) - a_{1} \left( a^{\p \p}_{2} a^{\p}_{0} - a^{\p \p}_{0} a^{\p}_{2}  \right) + a_{2} \left( a^{\p \p}_{1} a^{\p}_{0} - a^{\p \p}_{0} a^{\p}_{1}  \right) \Big) \\
& = 15 c \psi \cdot Wr_{2}(\tau) = 15 c^{\p} \psi \eta^{15}(\tau)
\end{align*}
where $c = 2 \left( 16 \pi i \right)^{3}$, $c^{\p} = c \left( \frac{\pi i}{4} \right)^{3} 2! 4!$  and $ Wr_{m}(\tau) = 2^{m-1} \det{ \left( \theta^{(\nu)}_{m,\mu}(\tau,0)_{0 \leq \nu, \mu \leq m} \right) } $ is the Jacobi-Theta Wronskian of order $2$. The equality $Wr_{2}(\tau) = \frac{\pi i}{4}^{3} 2! 4! \, \eta^{15}(\tau)$ ($\eta(\tau)$ being the Dedekind's $\eta$ - function) follows from \cite{kramer2}.

$D_{0}(6) \mid_{\ker{\pi_{1}}}$ is clearly injective. To show it's surjectivity it suffices to check that given $f \in S_{k+6}$, if we define $\psi$ by $\psi \cdot \eta^{15} := f$, then $\psi$ has the transformation properties~(\ref{transfpsi}). It is then easy to check that the equation~(\ref{psidefn}) defining the vector-valued modular form $(h_{0,1}, h_{0,2}, h_{1,2})$ gives it's transformation formulas from those of $\psi$ and $(a_{0}, a_{1}, a_{2})$ (see \cite[p.59]{zagier}), the conditions at infinity being trivially true. Since we already know that \[ \psi  \Big( a_{2} \left( \theta^{H}_{0,1} + \theta^{H}_{0,3} - \theta^{H}_{1,0} - \theta^{H}_{3,0} \right)  -2 a_{1} \left( \theta^{H}_{0,2} - \theta^{H}_{2,0} \right) + a_{0} \left( \theta^{H}_{1,2} - \theta^{H}_{2,1} - \theta^{H}_{2,3} + \theta^{H}_{3,2} \right)  \Big) \in \ker{ \pi_{1}}^{\circ} \] from the argument after equation~(\ref{transfpsi}), the assertion about sufficiency is true.

It remains to check the sufficiency. We know $\eta^{15} \left( -\frac{1}{\tau} \right) = \left( \frac{\tau}{i} \right)^{15} \eta(\tau), \, \eta^{15}(\tau + 1) = e^{\frac{5 \pi i}{4}} \eta(\tau)$. From the definition of $\psi$, we have \[ \psi \left( -\frac{1}{\tau} \right) \eta^{15} \left( -\frac{1}{\tau} \right) = \tau^{k+6} \psi( \tau) \eta^{15} (\tau), \q \psi( \tau +1) \eta^{15}(\tau + 1) = \psi(\tau) \eta^{15} (\tau) \]
which clearly gives the right transformation properties for $\psi$. From the definition of $\sigma$ (after the statement of Theorem~\ref{2mod4ind2}) we see that $\sigma$ induces an isomorphism between $ S_{k+2} \times S_{k+6} $ and $\ker{ \pi_{1}} $.

\textit{Step 2.} The case $k = 2$. We claim $J_{2,2}(\mathcal{O}_{K}) = 0$. Indeed, considering the restriction map $\pi_{1}$ we find $\dim{ J_{2,2}(\mathcal{O}_{K})} = \dim{\ker{ \pi_{1}}} + \dim{\mathrm{Im}( \pi_{1})} = \dim{S_{4}} + \dim{ S_{8}} = 0$ since $J_{2,2} = 0$ and using \textit{Step 1} for the description of $ \ker{ \pi_{1}}$. The Theorem is trivially true in this case.

\textit{Step 3.} $\ker{ D_{0}} = \mathrm{Im}( \pi_{1})$ for $k > 4$. From equations~(\ref{21}),~(\ref{22}),~(\ref{23}) it follows that 
\begin{align*}
D_{0} \circ \pi_{1} &= a_{0} H_{0} + 2 a_{1} H_{1} + a_{2} H_{2} \\
&=  a_{0} (2h_{0,1}a_{1} + h_{0,2}a_{2}) - 2 a_{1} (h_{0,1}a_{0} + h_{1,2}a_{2} ) - a_{2} (h_{0,2}a_{0} - 2h_{0,2}a_{1} )= 0 
\end{align*}
This could also be seen from the fact that $(D_{0} \circ \pi_{1}) \phi = D_{0}( \mc{O}_{K} ) \phi = \chi_{0,0} = 0$ since $\chi_{\alpha,\beta}=0$ unless $\alpha-\beta \equiv k \pmod{4}$. So $\mathrm{Im}( \pi_{1}) \subseteq \ker{ D_{0}}$. But a direct check (considering $k \equiv 2,6,10 \pmod{12}$) shows
\begin{align*}
 \dim{\mathrm{Im}( \pi_{1})} &= \dim{ J_{k,2}(\mathcal{O}_{K})} - \dim{\ker{ \pi_{1}} } = \left[ \frac{k-1}{3} \right] - \dim{S_{k+2}} -  \dim{S_{k+6}} \\
&=  \dim{J_{k,2}} - \dim{M_{k}} = \dim{\ker{ D_{0}}}
\end{align*}
 since $D_{0}$ is surjective. The dimension formula for $J_{k,2}^{cusp}(\mathcal{O}_{K})$ ($k > 4$) follows from \cite[Korollar 8.11, p.92]{haverkamp} or \cite[Theorem 3]{haverkamp/en} and the fact that $\dim{J_{k,2}^{Eis}(\mathcal{O}_{K})} = 0$ for $k \equiv 2 \pmod{4}$ (see \cite[Satz 2.5, p.25]{haverkamp}) gives $\dim{ J_{k,2}(\mathcal{O}_{K})} = \left[ \frac{k-1}{3} \right] $. 
This completes the proof of the Theorem.
\end{proof}

\subsection{$ \eta = \eta_{3}$}  \label{4.3}  In this case, it is easy to see that $ h_{0,0} = h_{2,2} = h_{0,2} = h_{2,0} = 0$, and after a calculation,
\begin{align}
& h_{0,3}= -h_{0,1},\q h_{1,0}= i h_{0,1},\q h_{1,3}= i h_{1,1},\q h_{2,1}= -i h_{1,2},\q h_{2,3}= i h_{1,2} \label{1mod42}\\
& h_{3,0}= -i h_{0,1},\q h_{3,1}= -i h_{1,1},\q h_{3,2}= - h_{1,2},\q h_{3,3}= -h_{1,1} .\label{1mod43}
\end{align}
Exactly the same argument as in the case $\eta = \eta_{1}$ works here, i.e., consider the map $\pi_{1+i} \colon J_{k,2}^{\eta_{3}}(\mathcal{O}_{K}) \rightarrow J_{k,4}$. Using Lemma~\ref{2to4} we have
\begin{align}
& \pi_{1+i} \phi(\tau,z) = \underset{\mu \pmod{8}}\sum \bar{H}_{\mu} \theta_{4,\mu}(\tau,z) \q \mbox{ where } \bar{H}_{0}= \bar{H}_{4}= 0 \\
& \bar{H}_{1}= -(1-i) h_{0,1} b_{1}  - (1+i) h_{1,2} b_{3}, \bar{H}_{2} = i h_{1,1} (b_{0} - b_{4}), \\
& \bar{H}_{3}= (1-i) h_{0,1} b_{3}  + (1+i) h_{1,2} b_{1},
\end{align}
from which we conclude that $\pi_{1+i}$ is injective. Also from the dimension formula for $k > 4$ we get $\dim{ J_{k,2}(\mathcal{O}_{K}) } = \frac{k-5}{4}$ ($ = \dim{ J_{k,4}}$ from \cite{zagier}), whereas $J_{1,2}(\mathcal{O}_{K}) \hookrightarrow  J_{1,4} =0$. Hence

\begin{prop} \label{1mod4ind2}
Let $k \equiv 1 \pmod{4}$. Then $\pi_{1+i}$ induces an isomorphism between $J_{k,2}(\mathcal{O}_{K}) $ and $J_{k,4}$.
\end{prop}

\subsection{$ \eta = \eta_{0}$}  \label{4.4} In this case $k \equiv 1 \pmod{4}$ and,
\begin{align}
& h_{0,1}= h_{0,3} =  h_{1,0}=  h_{3,0},\q h_{0,2}=  h_{2,0} \\
& h_{1,2}=  h_{2,1} = h_{2,3}=  h_{3,2},\q h_{1,1}=  h_{1,3} = h_{3,1} = h_{3,3} .
\end{align}

For $k > 4$ from \cite[Satz 2.5, p.25]{haverkamp} we find $\dim{J_{k,2}^{Eis}(\mathcal{O}_{K})} = 2$, and from \cite[Korollar 8.1, p.92]{haverkamp} or \cite[Theorem 3]{haverkamp/en} that $ \dim{J_{k,2}^{cusp}(\mathcal{O}_{K})} = \frac{k-4}{2}$ via the Trace formula for Hecke Operators. Therefore $\dim{J_{k,2}(\mathcal{O}_{K})} = \frac{k}{2}$.

We prove a Lemma which will be used in the proof of the next Theorem.

\begin{lem} \label{Utheta}
Let $\phi \in J_{k,1}(\mathcal{O}_{K})$ with Theta decomposition $\phi = h_{0} \theta_{1;0}^{H} +  h_{\frac{1}{2}} \theta_{1; \frac{1}{2}}^{H} + h_{\frac{i}{2}} \theta_{1; \frac{1}{2}}^{H} + h_{\frac{1+i}{2}} \theta_{1; \frac{1}{2}}^{H}$. Then the Theta decomposition of $U_{1+i} \phi$ is given by:
\begin{align}
U_{1+i} \phi = h_{0} \theta_{0,0}^{H} + h_{0} \theta_{2,2}^{H} + h_{\frac{1}{2}} \theta_{1,1}^{H} + h_{\frac{1}{2}} \theta_{3,3}^{H} + h_{\frac{i}{2}} \theta_{3,1}^{H} + h_{\frac{i}{2}} \theta_{1,3}^{H} + h_{\frac{1+i}{2}} \theta_{0,2}^{H} + h_{\frac{1+i}{2}} \theta_{2,0}^{H}
\end{align}
\end{lem}

\begin{proof}
First we note $(1+i) \mc{O}_{K} = 2 \mc{O}_{K} \cup (1+i) + 2 \mc{O}_{K} $ (disjoint union) as abelian groups. Let $s = \frac{x}{2} + i \frac{y}{2} \in \mc{S}_{2}$. We have 
\begin{align}
U_{1+i} \theta_{1, s}^{H} (\tau , z_{1} , z_{2}) &= \underset{r \equiv s \pmod{\mc{O}_{K}} }\sum e \left( N(r)\tau + (1+i)r z_{1} + (1-i)\bar{r} z_{2} \right) \nonumber \\
&= \underset{r^{\p} \equiv \frac{x-y}{2} + i \frac{x+y}{2} \pmod{(1+i)\mc{O}_{K}}}\sum e \left( \frac{N(r^{\p})}{2}\tau + r^{\p} z_{1} + \bar{r}^{\p} z_{2} \right) \nonumber
\end{align} 

Using the above formula and that $(1+i) \mc{O}_{K} = 2 \mc{O}_{K} \cup (1+i) + 2 \mc{O}_{K} $, we see that 
\begin{align}
U_{1+i} \theta_{1; 0}^{H} = \theta_{0, 0}^{H} + \theta_{2, 2}^{H}; \qq U_{1+i} \theta_{1;\frac{1}{2} }^{H} = \theta_{1, 1}^{H} + \theta_{3, 3}^{H}; \\
U_{1+i} \theta_{1;\frac{i}{2} }^{H} = \theta_{3, 1}^{H} + \theta_{1, 3}^{H}; \qq U_{1+i} \theta_{1;\frac{1+i}{2} }^{H} = \theta_{0, 2}^{H} + \theta_{2, 0}^{H}.
\end{align}
The lemma now follows at once.
\end{proof}

\begin{lem} \label{4,2} Let  $\phi_{4,1}$ be the basis element of $J_{4,}(\mathcal{O}_{K})$ given in \cite{sasaki}. Then,

$ \left\{ U_{1+i} \phi_{4,1}, \, \phi_{4,1} \mid V_{2} \right\}$ is a basis of $J_{4,2}(\mathcal{O}_{K})$.
\end{lem}

\begin{proof}
The Taylor expansion of $\phi_{4,1}$ around $z_{1} = z_{2} =0$ is
$ \phi_{4,1}(\tau, z_{1} , z_{2}) = 2 E_{4} + \pi i E_{4}^{\prime} z_{1}  z_{2} + \cdots $
from which the proof follows easily by writing down the corresponding Taylor expansions of $U_{1+i} \phi_{4,1}$ and $\phi_{4,1} \mid V_{2}$.
\end{proof}

\begin{thm} \label{0mod4ind2}
Let $k \equiv 0 \pmod{4}$. We have the following exact sequence of vector spaces
\begin{equation}
0 \xrightarrow{} J_{k,2}(\mathcal{O}_{K}) \xrightarrow{\pi_{1} \times \pi_{1+i}} J_{k,2} \times J_{k,4} \xrightarrow{\Lambda(2) - \Lambda(4)} M_{k} \times S_{k+2} \xrightarrow{} 0 
\end{equation}
where $\Lambda(m) := D_{0} + \frac{2}{m} D_{2} \colon J_{k,m} \rightarrow M_{k} \times S_{k+2}$.
\end{thm}

\begin{proof}
We divide the proof into three steps.

\textit{Step 1.} Let $\phi \in \ker{\pi_{1} \times \pi_{1+i}}$. We invoke Lemma~\ref{2to2} and ~\ref{2to4}. Keeping the same notation as those in the Lemmas, we get $\pi_{1+i} \phi = \underset{\mu \in \mbb{Z}/ 8 \mbb{Z}}\sum \bar{H}_{\mu}(\tau) \cdot \theta_{4,\mu}(\tau,z)$ where $\bar{H}_{\mu } = \bar{H}_{- \mu}$ ($\mu \in \mbb{Z} / 8 \mbb{Z}$) and
\begin{align}
& \bar{H}_{0} = h_{0,0} b_{0}  + 2 h_{1,1} b_{2} + h_{2,2} b_{4} =0  \label{1}\\
& \bar{H}_{1} = 2 h_{0,1} b_{1} + 2 h_{1,2} b_{3}= 0 \label{2} \\
& \bar{H}_{2} = 2 h_{0,2} b_{2}  + 2 h_{1,1} b_{4}= 0\label{3}\\
&  \bar{H}_{3} = 2 h_{0,1} b_{3} + 2 h_{1,2} b_{1}= 0 \label{4} \\
& \bar{H}_{4} = h_{0,0} b_{4}  + 2 h_{1,1} b_{2} + h_{2,2} b_{0}=0 \label{5}
\end{align}
since $b_{\mu} = b_{- \mu}$. Further 
\begin{align}
& H_{0} = h_{0,0} a_{0}  + 2 h_{0,1} a_{3} + h_{0,2} a_{2} =0 \label{6}   \\
& H_{1} = h_{0,1} a_{0}  + 2 h_{1,1} a_{1}  + h_{1,2} a_{2} =0 \label{7}   \\
& H_{2} = h_{0,2} a_{0}  + 2 h_{1,2} a_{1}  + h_{2,2} a_{2} =0 \label{8} 
\end{align}
From~(\ref{2}) and~(\ref{4}) we get $( b_{1}^{2} - b_{3}^{2}) h_{0,1} = ( b_{1}^{2} - b_{3}^{2}) h_{1,2} = 0$. 

We claim that $\theta_{m,\mu}(\tau,0)  \neq \theta_{m,\nu}(\tau,0)$ for $ \mu \neq \nu$ ($0 \leq \mu , \nu \leq m$), $\tau \in \mc{H}$. Suppose not. Then the Wronskian $Wr_{m}$ of $\theta_{m,\mu}$ ($0 \leq \mu \leq m$), would be identically zero on $\mc{H}$ contradictiing the fact that it is a non-zero multiple of Dedekind's $\eta$- function \cite{kramer2}.    

Therefore $ b_{1}^{2}(\tau) \neq  b_{3}^{2} (\tau)$ for all $\tau \in \mc{H}$ which implies $h_{0,1} = h_{1,2}= 0$ (only $ b_{1}^{2} \not \equiv  b_{3}^{2}$ would have sufficed to get this conclusion). Finally~(\ref{7}) and~(\ref{3}) together imply $h_{1,1} = h_{0,2} = 0$. From~(\ref{1}) and~(\ref{5}) we get $( b_{0}^{2} - b_{4}^{2}) h_{0,0} = ( b_{0}^{2} - b_{4}^{2}) h_{2,2} = 0$. By the above, we get $h_{0,0} = h_{2,2} = 0$. Hence $\phi = 0$.

\textit{Step 2.} $\mathrm{Im}\left( \pi_{1} \times \pi_{1+i}\right) \subseteq \ker{ \left( \Lambda(2) - \Lambda(4) \right)}$. We use the Taylor expansions of the Jacobi forms involved. Let $ \phi(\tau,z_{1},z_{2})= \underset{\alpha,\beta \geq 0}\sum \chi_{\alpha,\beta}(\tau) z_{1}^{\alpha} z_{2}^{\beta} \in J_{k,2}(\mc{O}_{K})$ be the Taylor expansion of $\phi$ around $z_{1}=z_{2}=0$. Then the Taylor development of $\pi_{1} \phi $ and $\pi_{1+i} \phi$ are 
\begin{align}
&\pi_{1} \phi = \chi_{0,0} + \chi_{1,1} \, z^{2} + \left( \chi_{0,4} + \chi_{2,2} + \chi_{4,0} \right) z^{4} + \cdots \\
&\pi_{1+i} \phi =  \chi_{0,0} + 2 \chi_{1,1} \, z^{2} - 4 \left( \chi_{0,4} - \chi_{2,2} + \chi_{4,0} \right) z^{4} + \cdots 
\end{align}
from which it easily follows $\Lambda(2) \pi_{1} \phi = \Lambda(4) \pi_{1+i} \phi$.

$\Lambda(2)  - \Lambda(4)$ is clearly surjective since $\Lambda(2)$ is surjective (Recall that $D_{0} + D_{2} +  D_{4} \colon J_{k,2} \rightarrow M_{k} \times S_{k+2} \times S_{k+4}$ is an isomorphism).

$\mathrm{Im}\left( \pi_{1} \times \pi_{1+i}\right) = \ker{ \left( \Lambda(2) - \Lambda(4) \right)}$. We show that they have the same dimension (for  $k \geq 4$). First of all we have,
\begin{align*} 
& \dim{\mathrm{Im}\left( \pi_{1} \times \pi_{1+i}\right)} = \dim{J_{k,2}(\mathcal{O}_{K})} = \frac{k}{2} 
\end{align*}
(for $k >4$, use Hverkamp's dimension formula, see Lemma~\ref{4,2} for $k=4$). Whereas,
\begin{align*}
& \dim{\ker{ \left( \Lambda(2) - \Lambda(4) \right)}} = \dim{J_{k,2}} + \dim{J_{k,4}} - \dim{M_{k}} - \dim{S_{k+2}}
\end{align*}
From part $2$ of \thmref{0mod4} and the fact $\dim{J_{k,1}} = \frac{k}{4}$ (for $k \equiv 0 \pmod{4}$, see \cite[Theorem 1]{sasaki}) or computing directly we get $\dim{J_{k,2}} = \frac{k}{4}$. A direct check now shows 
\[ \dim{J_{k,4}} - \dim{M_{k}} - \dim{S_{k+2}} = \dim{S_{k+4}} + \dim{S_{k+6}} + \dim{S_{k+8}} = \frac{k}{4} \] (Recall that $D_{0} + D_{2} +  D_{4} + D_{6} + D_{8} \colon J_{k,4} \rightarrow M_{k} \times S_{k+2} \times S_{k+4} \times S_{k+6}\times S_{k+8} $ is an isomorphism).

This completes the proof of \thmref{0mod4ind2}.
\end{proof}

\noindent Next, we give the explicit Theta decompositions of two particular basis elements of $J_{4,2}(\mathcal{O}_{K})$. One could perhaps write down the Theta decomposition of $\phi_{4,1} \mid V_{2}$ from the corresponding decomposition of $\phi_{4,1}$, but we use a different trick.

\begin{prop}
$J_{4,2}(\mathcal{O}_{K})$ is spanned by the two linearly independent elements ${\Phi}_{4,2}$ and $\tilde{\Phi}_{4,2}$ given by 
\begin{align} \label{Phi4,2}
\Phi_{4,2}= (x^{6} + y^{6}) (\theta^{H}_{0,0} + \theta^{H}_{2,2}) + z^{6} ( \theta^{H}_{1,1} + \theta^{H}_{3,3} + \theta^{H}_{1,3} + \theta^{H}_{3,1}) + (x^{6} - y^{6}) (\theta^{H}_{0,2} + \theta^{H}_{2,0}) \q \text{and},
\end{align}
{\small\begin{align} \label{tildePhi4,2}
\tilde{\Phi}_{4,2} =  2 x^{3} y^{3} (\theta^{H}_{0,0} - \theta^{H}_{2,2}) +  z^{3}(x^{3} - y^{3}) (\theta^{H}_{0,1} + \theta^{H}_{1,0} + \theta^{H}_{0,3} + \theta^{H}_{3,0} ) + z^{3}(x^{3} + y^{3}) (\theta^{H}_{1,2} + \theta^{H}_{2,1} + \theta^{H}_{2,3} + \theta^{H}_{3,2} ).
\end{align}}

\end{prop}

\begin{proof}
Let $\Phi_{4,1}$ be a basis of $J_{4,1}(\mathcal{O}_{K})$ explicitly given in \cite{sasaki}:
\begin{align}
\Phi_{4,1} = (x^{6} + y^{6}) \theta_{1,0}^{H} + z^{6} ( \theta_{1,\frac{1}{2}}^{H} +\theta_{1,\frac{i}{2}}^{H}) + (x^{6} - y^{6}) \theta_{1,\frac{1+i}{2}}^{H}
\end{align} 
where $x = \underset{n \in \mbb{Z}}\sum e\left( \frac{n^{2} \tau}{2} \right), \q y = \underset{n \in \mbb{Z}}\sum (-1)^{n} e\left ( \frac{n^{2} \tau}{2} \right), \q z = \underset{t \in \frac{1}{2} + \mbb{Z}}\sum e\left( \frac{t^{2} \tau}{2} \right)$ are the so called ``Theta constants''. 

Let $\Phi_{4,2} := U_{1+i} \Phi_{4,1}$. We will produce another element of $J_{4,2}(\mathcal{O}_{K})$ linearly independent of $\Phi_{4,2}$. To this end, we compute the Theta decomposition of $\Phi_{4,2}$ using Lemma~\ref{Utheta} and the fact that $h_{\frac{1}{2}} = h_{\frac{i}{2}}$ ($h_{s}$ being Theta components of an element in $J_{k,1}(\mathcal{O}_{K})$, $k \equiv 0 \pmod{4}$) :
\begin{align} \label{Phi4,2}
\Phi_{4,2}= (x^{6} + y^{6}) (\theta^{H}_{0,0} + \theta^{H}_{2,2}) + z^{6} ( \theta^{H}_{1,1} + \theta^{H}_{3,3} + \theta^{H}_{1,3} + \theta^{H}_{3,1}) + (x^{6} - y^{6}) (\theta^{H}_{0,2} + \theta^{H}_{2,0}).
\end{align}

Here we consider the restriction $\pi_{1}$. Since $\dim{J_{4,2}} = 1$, and $\pi_{1}$ is non-zero, we also have $\pi_{1}$ is surjective. Since $\Phi_{4,2} \not \in \ker{\pi_{1}}$, $\dim{\ker{\pi_{1}}} = 1$. We will determine $\tilde{\Phi}_{4,2} \neq 0$ by the condition $\tilde{\Phi}_{4,2} - \Phi_{4,2} \in \ker{\pi_{1}}$, which will prove the Proposition.

Let $\phi \in J_{4,2}(\mathcal{O}_{K})$. The transformation formulas for it's Theta components under $S$ are as follows:
\begin{align}
&h_{0,0} \mid_{3} S = \frac{i}{4} ( h_{0,0} + h_{2,2} + 2 h_{0,2} + 4 h_{1,1} +  4 h_{0,1} + 4 h_{1,2} ) \label{62}\\
&h_{2,2} \mid_{3} S = \frac{i}{4} ( h_{0,0} + h_{2,2} + 2 h_{0,2} + 4 h_{1,1} -  4 h_{0,1} - 4 h_{1,2} ) \label{63}\\
&h_{0,1} \mid_{3} S = \frac{i}{4} ( h_{0,0} - h_{2,2} + 2 h_{0,1} - 2 h_{1,2} )\label{64} \\
&h_{1,2} \mid_{3} S = \frac{i}{4} ( h_{0,0} - h_{2,2} - 2 h_{0,1} + 2 h_{1,2} ) \label{65}\\
&h_{0,2} \mid_{3} S = \frac{i}{4} ( h_{0,0} + h_{2,2} + 2 h_{0,2} - 4 h_{1,1} ) \label{66}\\
&h_{1,1} \mid_{3} S = \frac{i}{4} ( h_{0,0} + h_{2,2} - 2 h_{0,2}  ) \label{67}
\end{align}  

Let us denote the Theta components of $\Phi_{4,2}$ by $\hat{h}_{s}$ ($s \in \mc{S}_{2}$). From~(\ref{Phi4,2}) we see that 
$\hat{h}_{0,1} = \hat{h}_{1,2} = 0$ which implies by equations~(\ref{62}),...,~(\ref{67}) that $\hat{h}_{0,0} = \hat{h}_{2,2} $ and the following transformation formulas under $S$:
\begin{align}
&\hat{h}_{0,0} \mid_{3} S = \frac{i}{2} (  \hat{h}_{0,0} + \hat{h}_{0,2} +  2 \hat{h}_{1,1} ) \label{71}\\
&\hat{h}_{0,2} \mid_{3} S = \frac{i}{2} (  \hat{h}_{0,0} + \hat{h}_{0,2} -  2 \hat{h}_{1,1} ) \label{72}\\
&\hat{h}_{1,1} \mid_{3} S = \frac{i}{2} (  \hat{h}_{0,0} - \hat{h}_{0,2} ) \label{73} 
\end{align}

From the above formulas~(\ref{62}),...,~(\ref{67}) we note that if we assume $\tilde{h}_{1,1} = \tilde{h}_{0,2} = 0$ in a Hermitian Jacobi form $\tilde{\phi}$ of weight $4$ and index $2$, (conditions complementary to that in $\Phi_{4,2}$), we get $\tilde{h}_{0,0} + \tilde{h}_{2,2} = 0$ and a ``honest'' vector-valued modular form of weight $3$:
\begin{align}
&\tilde{h}_{0,0} \mid_{3} S = \frac{i}{2} (  \tilde{h}_{0,1} +  \tilde{h}_{1,2} ) \label{68}\\
&\tilde{h}_{0,1} \mid_{3} S = \frac{i}{2} ( \tilde{h}_{0,0}  +  \tilde{h}_{0,1} -  \tilde{h}_{1,2} ) \label{69}\\
&\tilde{h}_{1,2} \mid_{3} S = \frac{i}{2} ( \tilde{h}_{0,0} -  \tilde{h}_{0,1} +  \tilde{h}_{1,2} ),\label{70}
\end{align}
the transformation formulas under $T$ remaining the same. Therefore Theorem~\ref{theta} will give a Hermitian Jacobi form of weight $4$ index $2$ with the above Theta components, which we denote by $\tilde{\Phi}_{4,2}$. If $\tilde{\Phi}_{4,2}$ exists and is non-zero, we are done, since by construction it is linearly independent of $\Phi_{4,2}$.

We will determine $\tilde{\Phi}_{4,2}$ by imposing the condition that $\tilde{\Phi}_{4,2} - \Phi_{4,2} \in \ker{\pi_{1}}$, i.e., $\pi_{1} \tilde{\Phi}_{4,2} = \pi_{1} \Phi_{4,2}$. Upon using Lemma~\ref{2to2} and equation~(\ref{Phi4,2}), the Theta components $\tilde{h}_{s}$ of $\tilde{\Phi}_{4,2}$ satisfy the following system of equations:
\begin{align}
&\tilde{h}_{0,0} a_{0} + 2 \tilde{h}_{0,1} a_{1} = (x^{6}+y^{6}) a_{0} + (x^{6}-y^{6}) a_{2} \\
&\tilde{h}_{0,1} a_{0} + \tilde{h}_{1,2} a_{2} = 2 z^{6} a_{1} \\
-&\tilde{h}_{0,0} a_{2} + 2 \tilde{h}_{1,2} a_{1} = (x^{6}-y^{6}) a_{0} + (x^{6}+y^{6}) a_{2}
\end{align}

Using the relations $x = a_{0} + a_{2}$,  $y = a_{0} - a_{2}$, $z = 2 a_{1}$, we find after a calculation,
\begin{align} \label{newthetacomp}
\tilde{h}_{0,0} = 2 x^{3} y^{3}, \q \tilde{h}_{0,1} = z^{3}(x^{3} - y^{3}), \q \tilde{h}_{1,2} = z^{3}(x^{3} + y^{3}).
\end{align}

It is now easy to see (using the formulas in \cite[p.59]{zagier}) that $\tilde{h}_{0,0}, \tilde{h}_{0,1}, \tilde{h}_{1,2}$ given by equation~(\ref{newthetacomp}) satisfy the right transformation formulas~(\ref{68}), (\ref{69}), (\ref{70}) under $S$. Whence, $\tilde{\Phi}_{4,2}$ is non-zero and has the Theta decomposition:
{\small\begin{align} \label{tildePhi4,2}
\tilde{\Phi}_{4,2} =  2 x^{3} y^{3} (\theta^{H}_{0,0} - \theta^{H}_{2,2}) +  z^{3}(x^{3} - y^{3}) (\theta^{H}_{0,1} + \theta^{H}_{1,0} + \theta^{H}_{0,3} + \theta^{H}_{3,0} ) + z^{3}(x^{3} + y^{3}) (\theta^{H}_{1,2} + \theta^{H}_{2,1} + \theta^{H}_{2,3} + \theta^{H}_{3,2} )
\end{align}}

As noted earlier, this completes the proof of the Proposition.
\end{proof}

\subsection{Order of vanishing at origin} \label{originorder} 
For $\phi \in J_{k,m}(\mathcal{O}_{K})$ let $\phi(\tau,z_{1},z_{2}) = \underset{\alpha,\beta \geq 0}\sum \chi_{\alpha,\beta}(\tau) z_{1}^{\alpha}z_{2}^{\beta}$ be the Taylor expansion around $z_{1} = z_{2} = 0$. Define a non-negative integer $\varrho_{k,m} \phi$ by
\begin{align}
\varrho_{k,m} \phi &= \min{ \left\{ \alpha+\beta \mid \chi_{\alpha,\beta}(\tau) \not \equiv 0  \right\} } &\mbox{if} \, \phi \not \equiv 0  \\
& = \infty &\mbox{otherwise}
\end{align}
i.e., $\varrho_{k,m} \phi$ can be interpeted as the order of vanishing of $\phi$ at the origin. From the relation with Jacobi forms, we can give upper bounds on $\varrho_{k,m} \phi$ for any $\phi \in J_{k,m}(\mathcal{O}_{K})$ ($m= 1,2$).

\begin{prop} \label{orderofvanish}
(i) Let $\phi \in J_{k,1}(\mathcal{O}_{K})$ be non zero. Then 
\begin{align}
0 \leq \varrho_{k,1} \leq 2 \, \mbox{ if } k \equiv 2 \pmod{4} \, ; \q  0 \leq \varrho_{k,1} \leq 4 \, \mbox{ if } k \equiv 0 \pmod{4}
\end{align}
(ii) Let $\phi \in J_{k,2}(\mathcal{O}_{K})$. Then 
\begin{align}
0 \leq \varrho_{k,2} \leq 5 \, \mbox{ if } k \equiv 1,3 \pmod{4} \, ; \q 0 \leq \varrho_{k,2} \leq 8 \, \mbox{ if } k \equiv 0,2 \pmod{4}
\end{align}
\end{prop}

\begin{proof}
All of these assertions except the case $k \equiv 2 \pmod{4}, \, m = 2$ follow easily from Propositions~\ref{1mod4ind2},
 \ref{3mod4ind2} and Theorem~\ref{0mod4ind2} and the corresponding result for Jacobi forms (see \cite[p.37]{zagier}). In the case $k \equiv 2 \pmod{4}, \, m = 2$ we have $J_{k,2}(\mathcal{O}_{K})  \xrightarrow{ \pi_{1} \times \pi_{1+i} } J_{k,2} \times J_{k,4}$ (as in the case $k \equiv 0 \pmod{4}$). This follows from Lemmas~\ref{2to2} and~\ref{2to4} and equations~(\ref{2mod41}) and~(\ref{2mod42}). For convenience, we give the proof. Let $ \phi \in J_{k,2}(\mathcal{O}_{K})$, with Theta components $h_{a,b}$ ($a,b \in \mc{S}_{2}$). From $\pi_{1+i} \phi = 0$, we get $h_{1,1} = 0$ and from $\pi_{1} \phi = 0$ that (using $h_{0,0} = h_{2,2} = 0$)
\[  \begin{pmatrix}
2a_{1} & a_{2} & 0\\
a_{0} & 0 & a_{2} \\
0& a_{0} & 2a_{1} \end{pmatrix} 
\begin{pmatrix} 
h_{0,1}\\
h_{0,2}\\
h_{1,2} \end{pmatrix} = 0 \]
Since $\det{ \left( \begin{smallmatrix}
2a_{1} & a_{2} & 0\\
a_{0} & 0 & a_{2} \\
0& a_{0} & 2a_{1} \end{smallmatrix} \right) }  = - 4 a_{0} a_{1} a_{2}$, we get the required injectivity and hence the Proposition.
\end{proof}

\section{Rank of $J_{n*,m}(\mathcal{O}_{K})$ over $M_{*}$ and algebraic independence of $\phi_{4,1}$, $\phi_{8,1}$, $\phi_{12,1}$} \label{rank}

We refer to \cite{sasaki} for the definition of the index $1$ forms $\phi_{4,1}$, $\phi_{8,1}$, $\phi_{12,1}$ which are a basis for $J_{4*,1}(\mathcal{O}_{K}) := \underset{k \geq 0} \oplus J_{4k,1}(\mathcal{O}_{K})$ as a module over $M_{*}$.

\begin{rmk} \label{moduledef}

For $ m \geq 1$, $J_{n*,m}(\mathcal{O}_{K})$ ($n= 2,4$) are modules over $M_{*}$ via the algebra isomorphism
\begin{equation}
M_{*} = \mbb{C}[E_{4}, E_{6}] \xrightarrow{E_{4} \mapsto E_{4}, E_{6} \mapsto E_{6}^{2}} \mbb{C}[E_{4}, E_{6}^{2}] 
\end{equation}

because $E_{6} \cdot J_{n*,m}(\mathcal{O}_{K}) \not \in J_{n*,m}(\mathcal{O}_{K})$. From the argument in \cite[p.97]{zagier}, we easily see that $J_{*,*}(\mathcal{O}_{K})$ is free over $M_{*}$, and $J_{n*,m}(\mathcal{O}_{K})$ is of finite rank $r_{n}(m)$ over $M_{*}$.

\end{rmk}

\begin{prop}

(i) $r_{4}(m) = m^{2} + 2$, \q (ii) $r_{2}(m) = 2(m^{2} +1)$.
\end{prop}

\begin{proof}

The proof is immediate from the dimension formula of $J_{k,m}(\mathcal{O}_{K})$ in \cite[Theorem 3]{haverkamp/en}. We find that $\dim { J_{k,m}(\mathcal{O}_{K}) } = (m^{2} + 2) \dim{M_{k}} + f(m) + O(1)$, (resp. $= m^{2} \dim{M_{k}} + g(m) + O(1)$) where $f(m), g(m)$ are functions depending only on $m$ when $k \equiv 0 \pmod{4}$ (resp. $k \equiv 2 \pmod{4}$). Letting $k \rightarrow \infty$, we get $(i)$. Since there can be no linear relation between generators of weights $0 \pmod{4}$ and  
$2 \pmod{4}$ by Remark~\ref{moduledef}, $(ii)$ follows. 
\end{proof}

\begin{prop} \label{halfgen}

$\phi_{4,1}$, $\phi_{8,1}$, $\phi_{12,1}$ are algebraically independent over $M_{*}$.

\end{prop}

\begin{proof}

It is enough to prove the algebraic independence of $ \psi_{8,1}, \tilde{\psi}_{16,1}, \psi_{16,1}$, where $\psi_{8,1} = E_{4} \phi_{4,1} - \phi_{8,1}, \, \psi_{12,1} = E_{4} \phi_{8,1} - \phi_{12,1}, \, \tilde{\psi}_{16,1} = 5E_{4} \psi_{12,1} - 3 \psi_{16,1} $; $\psi_{8,1}, \psi_{12,1}, \psi_{16,1}$ being the generators of $J_{4*,1}^{cusp}(\mathcal{O}_{K})$ over $M_{*}$ (see \cite{sasaki} for their definition).

Let $f(X,Y,Z) =  \underset{a+b+c = m} \sum Q_{a,b,c} X^{a} Y^{b} Z^{c}$ be a homogeneous polynomial over $M_{*}$ of least degree $m$ such that $f(\psi_{8,1},\tilde{\psi}_{16,1},\psi_{16,1}) = 0$. Applying the map $\pi_{1}$ in the above relation we get \[ \underset{b+c = m} \sum Q_{0,b,c} \, (\pi_{1}\tilde{\psi}_{16,1})^{b} \, (\pi_{1} \psi_{16,1})^{c} = 0, \qq \mbox{ since } \pi_{1} \psi_{8,1} = 0. \]

From \lemref{res1} $\pi_{1}\tilde{\psi}_{16,1} \neq 0, \, D_{0} \, \pi_{1}\tilde{\psi}_{16,1} = 0, \, D_{0} \, \pi_{1} \psi_{16,1} \neq 0$. Hence, the argument in \cite[p.90]{zagier} for classical Jacobi forms applies, showing $Q_{0,b,c} = 0$ for all $b,c$ such that $b+c = m$. Hence we have \[ \underset{a+b+c = m, \, a \geq 1} \sum Q_{a,b,c} \, \psi_{8,1}^{a} \, \tilde{\psi}_{16,1}^{b} \,  \psi_{16,1}^{c}  = 0 , \] 
giving an equation of lower degree. Hence the Proposition is proved.
\end{proof}

\begin{lem} \label{res1}

(i) $\psi_{16,1} = - 2^{8} \Delta \phi_{4,1} \, + \, 2 E_{4}^{2} \phi_{8,1} \, - \, E_{4} \phi_{12,1}$.

(ii) $D_{0} \psi_{16,1} = -5 \cdot 2^{8} \Delta E_{4}$, \q $D_{0} \psi_{12,1} = -3 \cdot 2^{8} \Delta$.

(iii) $\pi_{1} \tilde{\psi}_{16,1} = (  \frac{2}{9} E_{4}^{3} \, + \, 3 \cdot 2^{9} \Delta ) \, E_{4,1} + \frac{8}{9} E_{4} E_{6} \, E_{6,1} $; $E_{4,1} , E_{6,1}$ being the normalised Jacobi Eisenstein series, which are a basis of $J_{2*,1}$ over $M_{*}$. $\Delta = q \prod_{n=1}^{\infty} ( 1 - q^{n})^{24}$ is the Discriminant cusp form of weight $12$.
\end{lem}

\begin{proof}
The calculations follow from the Theta decompositions of $\phi_{4i,1}$ ($i= 1,2,3$) given in \cite{sasaki} and using the Theta relations (see \cite{mumford}) : Let $ \theta_{s}^{H}(\tau) := \theta_{0,0}^{H} (\tau,0,0)$ ($s \in \mc{S}_{1}$). Then we have \[ \theta_{0,0}^{H}(\tau) = \frac{1}{2} (x^{2} + y^{2}), \, \theta_{0,1}^{H}(\tau) = \theta_{1,0}^{H}(\tau) = \frac{1}{2} z^{2} , \, \theta_{1,1}^{H}(\tau) = \frac{1}{2} (x^{2} - y^{2}); \] where \[ x = \underset{n \in \mbb{Z}}\sum e\left( \frac{n^{2} \tau}{2} \right), \q y = \underset{n \in \mbb{Z}}\sum (-1)^{n} e\left ( \frac{n^{2} \tau}{2} \right), \q z = \underset{t \in \frac{1}{2} + \mbb{Z}}\sum e\left( \frac{t^{2} \tau}{2} \right) \] are the ``Theta constants''. We omit the calculations.
\end{proof}

\section{Further questions and remarks} \label{qs}
\begin{enumerate}
\item
The restriction maps that we use in this paper do not commute with Hecke operators. Nevertheless it is expected that the following should be true. There should exist finitely many algebraic integers $\rho_{j} \in \mc{O}_{K}, \, 1 \leq j \leq n$ (where $n$ depends only on the index $m$) such that we have an embedding/isomorphism :
\[ \pi_{\rho_{1}} \times \cdots \times \pi_{\rho_{n}} \colon J_{k,m}(\mathcal{O}_{K}) \hookrightarrow J_{k,m N(\rho_{1})} \times \cdots \times J_{k,m N(\rho_{n})} \]
where $N : K \rightarrow \mbb{Q}$ is the norm map. From the results of this paper (see Theorem~\ref{2mod4ind2}, Proposition~\ref{orderofvanish}) this is true for $m=1,2$ and these cases suggest that $\rho_{j}$ and $n$ above should be related to the decomposition of $m$ in $\mc{O}_{K}$.

\item
We know that for $k \equiv 0 \pmod{4}$, $\dim{ J_{k,2}(\mathcal{O}_{K}) } = \frac{k}{2} = 2 ( \dim{ M_{k-4} } + \dim{ M_{k-8} } + \dim{ M_{k-12} })$. This suggests what the minimal weights of the $6$ generators of $J_{4*,2}(\mathcal{O}_{K})$ over $M_{*}$ should be, but the calculations seem to be much more than that in the case of classical Jacobi forms.

\item
It would be interesting to write down $m^{2} +2$ forms in $J_{4*,m}(\mathcal{O}_{K})$ linearly independent over $M_{*}$. We have $\binom{m+2}{2}$ such forms from Proposition~\ref{halfgen}. Since $m^{2} +2 - \binom{m+2}{2} = \binom{m-1}{2}$, we ask the following question. Let $A := \phi_{4,1}$, $B := \phi_{8,1}$, $C := \phi_{12,1}$. Does there exist a form $\phi_{3}$ in $J_{k,3}(\mathcal{O}_{K})$ (for suitable $k$) such that the set \[ \left\{ A^{a} \, B^{b} \, C^{c}  \right\}_{a+b+c = m} \cup \left\{ \phi_{3} \, A^{a} \, B^{b} \, C^{c}  \right\}_{a+b+c = m-3} \]
consists of $m^{2} +2$ linearly independent forms (over $M_{*}$) in $J_{4*,m}(\mathcal{O}_{K}) ? $ 
\end{enumerate}

\section{\textit{Acknowledgements}} The author wishes to thank Prof. B. Ramakrishnan for his support and encouragement and Prof. R. Sasaki for providing the paper \cite{sasaki}.

\end{document}